\documentclass[12pt]{amsart}
\usepackage{amssymb,amsthm,enumerate,wasysym}
\usepackage{array,arydshln}
\usepackage{pstricks, pst-plot, pst-node}
\usepackage{hyperref}                         
\usepackage{enumitem}                         
\usepackage[all]{xy}

\usepackage{tikz}                                 
\usetikzlibrary{angles}
\usetikzlibrary{decorations.markings,calc}
\usetikzlibrary{cd}

\usepackage[textwidth=1.8cm]{todonotes}
\reversemarginpar

\input xy
\xyoption{all}

\hypersetup{pdftex,                           
bookmarks=true,
pdffitwindow=true,
colorlinks=true,
citecolor=black,
filecolor=black,
linkcolor=black,
urlcolor=black,
hypertexnames=true}

\newcommand{\IZ}{{\mathbb{Z}}}

\newcommand{\fA}{{\mathfrak{A}}}
\newcommand{\fS}{{\mathfrak{S}}}

\newcommand{\bG}{{\mathbf{G}}}      

\newcommand{\cO}{{\mathcal{O}}}

\newcommand{\bB}{{\mathbf{B}}}

\newcommand{\bb}{{\mathbf{b}}}
\newcommand{\bc}{{\mathbf{c}}}

\newcommand{\dade}{{\mathbf{D}_k}}

\def\GL{ \text{\rm GL} }

\def\Sp{ \text{\rm Sp} }

\def\SL{ \text{\rm SL} }
\def\PSL{ \text{\rm PSL} }

\def\SU{ \text{\rm SU} }
\def\PSU{ \text{\rm PSU} }

\def\Spin{ \text{\rm Spin} }

\def\Br{ \text{\rm Br} }

\DeclareMathOperator{\End}{End}               
\DeclareMathOperator{\Res}{Res}               
\DeclareMathOperator{\Ind}{Ind}                  
\DeclareMathOperator{\Inf}{Inf}                    

\DeclareMathOperator{\Irr}{Irr}
\DeclareMathOperator{\IBr}{IBr}
\DeclareMathOperator{\Stab}{Stab}

\DeclareMathOperator{\Capp}{Cap}

\newcommand{\tw}[1]{{}^#1\!}

\let\lra=\longrightarrow
\let\wt=\widetilde
\let\wh=\widehat

\let\elll=l

\newtheorem{thm}{Theorem}[section]
\newtheorem{lem}[thm]{Lemma}

\newtheorem{cor}[thm]{Corollary}
\newtheorem{prop}[thm]{Proposition}

\newtheorem{quess}[thm]{Questions}

\theoremstyle{theorem}

\theoremstyle{definition}

\newtheorem{defn}[thm]{Definition}

\theoremstyle{remark}

\newcommand{\etalchar}[1]{$^{#1}$}

\begin{document}


\title[On the source algebra equivalence class of blocks with cyclic defect groups, I]
{On the source algebra equivalence class of blocks with cyclic defect groups, I}

\date{July 29, 2022}

\author{
Gerhard Hiss and  Caroline Lassueur
}
\address{{\sc Gerhard Hiss},  Lehrstuhl f\"ur Algebra und Zahlentheorie, RWTH Aachen 
University, 52056 Aachen, Germany.}
\email{gerhard.hiss@math.rwth-aachen.de}
\address{{\sc Caroline Lassueur}, FB Mathematik, TU Kaiserslautern, Postfach 3049, 
67653 Kaiserslautern, Germany.}
\email{lassueur@mathematik.uni-kl.de}



\begin{abstract} 
We investigate the source algebra class of a $p$-block with cyclic defect groups 
of the group algebra of a finite group.
By the work of Linckelmann this class is parametrized by the Brauer tree
of the block together with a sign function on its vertices and an endo-permutation module 
of a defect group. 
We prove that this endo-permutation module can be read off from the character 
table of the group. We also prove that this module is trivial for all cyclic $p$-blocks of 
quasisimple groups with a simple quotient which is a sporadic group, an alternating group,
a group of Lie type in defining characteristic, or a group of Lie type in cross-characteristic   
for which the prime~$p$ is large enough in a certain sense. 
\end{abstract}

\thanks{The second author gratefully acknowledge financial support by SFB TRR 195.}

\maketitle


\pagestyle{myheadings}
\markboth{On the source algebra equivalence class of blocks with cyclic defect groups, I}
{On the source algebra equivalence class of blocks with cyclic defect groups, I}

\vspace{6mm}
\section{Introduction}

Consider a finite group $G$, an algebraically closed field $k$ of characteristic 
$p>0$, and a $p$-block~$\bB$ of~$kG$ with a non-trivial cyclic defect group~$D$.
Let~$D_1$ denote the unique subgroup of~$D$ of order~$p$ and let~$\bb$ be the 
Brauer correspondent of~$\bB$ in~$N_G(D_1)$. For further terminology used below, 
we refer to Section~\ref{sec:notat}.

By Linckelmann's results \cite[Theorem~2.7]{Linck96}, the block~$\bB$ is 
determined up to source algebra equivalence by three invariants: 

{\rm (a)}  the Brauer tree $\sigma(\bB)$ of~$\bB$ with its planar embedding;

{\rm (b)}  the type function on $\sigma(\bB)$, which associates one of the 
        signs~$+$ or~$-$ to the vertices of $\sigma(\bB)$ in such a way that 
        adjacent vertices have different signs; and
        
{\rm (c)}  an indecomposable capped endo-permutation $kD$-module~$W(\bB)$ 
  isomorphic, by definition, to a $kD$-source of the simple $\bb$-mod\-ul\-es.

We observe that~$D_1$ acts trivially on~$W(\bB)$ and recall that (a) determines 
the Morita equivalence class of~$\bB$, whereas it is necessary to add 
parameters (b) and (c) to determine  the source algebra of~$\bB$ up to isomorphism
of interior $D$-algebras.

In~\cite[Section~$5$]{HL19}, the authors 
explicitly determined the structure of the indecomposable 
trivial source modules belonging to a cyclic $p$-block~$\bB$. This involves 
all three invariants (a), (b) and (c) above, and in particular
the module~$W( \bB )$ in an essential way. The formulae are particularly simple provided  
$\bB$ is uniserial or $W( \bB ) \cong k$.

It is known, nevertheless, that all indecomposable capped endo-permutation 
$kD$-modules on which $D_1$ acts 
trivially arise as a source of the simple $\bb$-modules for some cyclic 
block $\bB$; this follows from examples given by Dade in \cite{Dad66} 
(reinterpreted in terms of Linckelmann's definition of $W(\bB)$). Dade's 
construction was later extended and generalized by Mazza in~\cite{Mazza}. 
However, all the examples provided by the methods of~\cite{Dad66} and 
Mazza~\cite{Mazza} arise from $p$-blocks of $p$-solvable groups, and such blocks
are uniserial as $k$-algebras (i.e.\ all their projective indecomposable
modules are uniserial).

It is therefore natural to ask whether it is possible to describe all source
algebra equivalence classes of cyclic blocks arising in finite groups in function 
of the invariants (a), (b) and (c), and in particular which capped endo-permutation 
$kD$-modules can occur as~$W(\bB)$ for cyclic blocks~$\bB$ which are not uniserial. 
These questions are the main motivation behind the present article. 
More precisely, we are going to address the following questions.

\begin{quess}{\ }
\label{QuestionA}

{\rm (a)} Can we determine~$W(\bB)$ from the character table of~$G$?

{\rm (b)}  Can we determine~$W(\bB)$ for all non-uniserial cyclic blocks $\bB$ of 
              finite groups?
              
{\rm (c)} Can we determine~$W(\bB)$ for all cyclic blocks $\bB$ of quasisimple 
              groups?
              
{\rm (d)}  Can we reduce Question~{\rm (b)} to Question~{\rm (c)}?

\end{quess}
In Part~I of our paper, in Theorem~\ref{thm:readingWfromIrr(B)}, 
we give a positive answer to 
Question~\ref{QuestionA}(a) for odd primes~$p$. 
We also provide further criteria for $W(\bB)$ to be trivial. 
Then, in Section~\ref{sec:red} we  work towards reduction
theorems. In particular, we prove the analogues of Feit's reduction theorems
\cite{Fei84} for the determination of all Brauer trees. Our results, however,
do not give a complete reduction of Question~(b) to quasisimple groups.
Finally, in Section~\ref{sec:W(B)quasisimple}, 
we start the classification of the modules $W( \bB )$ (up to isomorphism) 
for the cyclic $p$-blocks of the quasisimple groups. We treat the quasisimple 
groups with a simple quotient which is a sporadic group, an alternating 
group, a group of Lie type in defining characteristic, or
a finite group of Lie type for which~$p$ is large in a certain sense. In all 
these cases, it turns out that $W( \bB ) \cong k$. However, this is not true in
general, although the possibilities for $W( \bB )$ are rather restricted.
The complete classification is fairly tedious and is postponed to 
Part~II of our paper.

Finally, we notice that the methods used in Section~\ref{sec:W(B)quasisimple} 
are of group-theoretical nature and show that in the treated cases all the 
centralizers of all the non-trivial $p$-elements lying in a given cyclic defect 
group are equal; this result is of independent interest.

\section{Notation and preliminaries}\label{sec:notat}

\subsection{General Notation}
Throughout this paper we let $p$ be a prime number and $G$ be a finite group of 
order divisible by~$p$. We let $(K,\cO,k)$ be a $p$-modular system, where~$\cO$ 
is a complete discrete valuation ring with field of fractions $K$ of 
characteristic zero and algebraically closed residue field $k$ of 
characteristic~$p$. We also assume that~$K$ is large enough in the sense 
that~$K$ is a splitting field for all subgroups of~$G$.
\par
Whenever $D$ denotes a cyclic $p$-subgroup of $G$ of order $p^{\elll}$ 
where $\elll$ is a positive integer, then for each $0\leq i\leq \elll$ we denote 
by~$D_{i}$ the cyclic subgroup of $D$ of order $p^{i}$ and we set $N_i:=N_G(D_i)$.
\par
Unless otherwise stated, $kG$-modules are assumed to be finitely generated left $
kG$-modules and by a $p$-block of~$G$ we mean a block of~$kG$. Given a 
subgroup $H\leq G$, we let $k$  denote the trivial $kH$-module, we write 
$\Res^G_H(M)$ for the restriction of the $kG$-module $M$ to~$H$, and 
$\Ind_H^G(N)$ for the induction of the $kH$-module $N$ to $G$. Given a normal 
subgroup $U$ of $G$, we write $\Inf_{G/U}^{G}(M)$ for the inflation of the 
$k[G/U]$-module $M$ to $G$. An analogous notation is used for characters. 
We let $\Omega$ denote the Heller operator. The canonical homomorphism 
$\pi:\cO\lra\cO/J(\cO)=k$ induces a bijection between the blocks of $\cO G$ 
and the blocks of $kG$; we simply denote by $\Irr_{K}(\bB)$ and $\IBr_{p}(\bB)$  
the set of irreducible $K$-characters, respectively the set of irreducible 
Brauer characters, of the preimage of a $p$-block~$\bB$ of~$G$ under $\pi$. 
\par
{We also recall that two $p$-blocks $\bB_{1}$ and $\bB_{2}$
of finite groups with a common defect group $D$ are called 
\emph{source algebra equivalent} (or \emph{Puig equivalent}, 
or \emph{splendidly Morita equivalent})  if there is an isomorphism of interior 
$D$-algebras between a source algebra of $\bB_{1}$ and a source 
algebra of $\bB_{2}$. Equivalently, a source algebra equivalence between 
$\bB_{1}$ and $\bB_{2}$ is a Morita equivalence between $\bB_{1}$ and $\bB_{2}$ 
which is induced by a pair of bimodules with trivial sources; 
see~\cite[Theorem 4.1]{Linck94}. 
}

\subsection{Radical $p$-subgroups}
By~$O_p(G)$ we denote the largest normal $p$-subgroup of~$G$. A $p$-subgroup 
$P \leq G$ is called \textit{radical}, if $O_{p}( N_G(P) ) = P$. A defect group 
of a $p$-block of~$G$ is a radical $p$-subgroup. The following lemmas are 
elementary and well-known, hence their proofs are 
omitted. Some of these results will only be used in Part~II of our paper.

\begin{lem}
\label{AbelianRadicalSubgroup}
Let $P \leq G$ be an abelian radical $p$-subgroup. Then $O_p( C_G( P ) ) = P = 
O_p( Z( C_G( P ) ) )$. Moreover, $N_G( P ) = N_G( C_G( P ) )$.
\end{lem}

\begin{lem}
\label{CentralExtensionEtc}
Let $Z \leq Z(G)$ and let $P$ be a $p$-subgroup of~$G$. For ${H \leq G}$, we 
write~$\bar{H}$ for the image of~$H$ under the canonical epimorphism 
$G \rightarrow G/Z$. Then the following hold.

{\rm (a)} If $p \nmid |Z|$, then $C_G( P )/Z = C_{\bar{G}}( \bar{P} )$.
If $O_p(Z) \leq P$, then $N_G(P)/Z = N_{\bar{G}}( \bar{P} )$.

{\rm (b)} Suppose that $N_G(P)/Z = N_{\bar{G}}( \bar{P} )$. Then~$P$
is a radical $p$-subgroup of~$G$ if and only if~$\bar{P}$ is a radical
$p$-subgroup of~$\bar{G}$.

{\rm (c)} Suppose that $N \unlhd G$ with $P \leq N$ such that 
$p \nmid [G\colon\!N]$.
Then~$P$ is a radical $p$-subgroup of~$N$ if and only if~$P$ is a radical 
$p$-subgroup of~$G$.
\end{lem}

\subsection{Endo-permutation modules}
Let $P$ be a finite $p$-group. A~$kP$-module $W$ is called 
\emph{endo-permutation} 
if $\End_k(W)$, endowed with the $kP$-module structure given by the conjugation 
action of $G$, is a permutation $kP$-module. Such an endo-permutation 
$kP$-module is called \emph{capped} if it has an indecomposable direct summand 
with vertex $P$. In this case, $W\cong \Capp(W)^{\oplus m}\oplus X$, where 
$m\geq1$ is an integer, all the indecomposable direct summands of $X$ have 
vertices strictly contained in $P$, and $\Capp(W)$, called the \emph{cap} of 
$W$, is, up to isomorphism, the unique indecomposable direct summand of $W$ 
with vertex $P$. The \emph{Dade group} of $P$, denoted $\dade(P)$, is the set of 
isomorphism classes of indecomposable capped endo-permutation $kP$-modules with 
composition law induced by the tensor product over $k$, i.e.\ 
\[
[W_{1}]+[W_{2}]:=[\Capp(W_{1}\otimes_k W_{2})]\,.
\]
If $Q\leq P$, then we denote by $\Omega_{P/Q}$ the relative Heller operator with 
respect to $Q$. 
By definition, if~$W$ is a $kP$-module, then $\Omega_{P/Q}(W)$ is the kernel of 
a relative $Q$-projective cover of~$W$. Moreover, $\Omega_{P/Q}(W)$ is 
an endo-permutation $kP$-module if and only if~$W$ is. The usual Heller 
operator is $\Omega=\Omega_{P/\{1\}}$. 
For a detailed introduction to endo-permutation modules we refer to the survey 
\cite[\S3-\S4]{TheSurvey} and the references therein.

\subsection{Blocks with cyclic defect groups}\label{ssec:cyclicBlocks}
From now on, if not otherwise specified, we denote by $\bB$ a $p$-block of $kG$ 
with a cyclic defect group $D\cong C_{p^\elll}$ with $\elll\geq 1$, and we let 
$\sigma(\bB)$ denote its Brauer tree, understood with its planar embedding.
We let~$\bb$ be the Brauer correspondent of $\bB$ in $N_{G}(D_{1})$ and~$\bc$ be 
a $p$-block of $C_{G}(D_1)$ covered by~$\bb$. Then $D$ is also a defect group of 
$\bb$ and~$\bc$. The inertial index of $\bb$ is equal to the inertial index 
of~$\bB$, whereas $\bc$ is nilpotent, that is, of inertial index one and has a 
unique non-exceptional character. The Brauer tree $\sigma(\bb)$ is a star with 
exceptional vertex at its center (if any). In particular, both $\bb$ and $\bc$ 
are uniserial blocks, in the sense that their projective indecomposable modules
are uniserial. Indeed, the structure of these modules is encoded in the Brauer 
tree; see \cite[Section~$17$]{AlperinBook}. The simple $\bb$-modules and the 
simple $\bc$-module have a common $kD$-source, which we denote by $W(\bB)$. 
Obviously, $W(\bB)=W(\bb)=W(\bc)$ and this is always an indecomposable capped 
endo-permutation $kD$-module, on which $D_{1}$ acts trivially, 
see~\cite[Theorem~11.1.5]{LinckBook}.

\section{On the determination of $W(\bB)$}

Throughout this section, we assume that~$p$ is a prime and $\bB$ is a 
$p$-block of $kG$ with a cyclic defect group $D\cong C_{p^\elll}$ 
where~$\elll\geq 1$, and associated capped endo-permutation $kD$-module~$W(\bB)$.
For further notation, see Section~\ref{sec:notat}.

\subsection{Recognizing $W(\bB)$ from the character table}
Here we assume that~$p$ is odd.
To begin with, we need to describe the values of the characters of the lifts to $\cO$ of
the indecomposable capped endo-permutation $kD$-modules. The latter modules were 
classified by Dade in \cite{DadeI&II}. Namely, 
\begin{equation}\label{eq:dadegroupcyclic}
\dade(D)=\langle \Omega_{D/D_i}(k) \mid 0\leq i \leq \elll-1 \rangle 
\cong 
(\IZ/2)^{\elll}\,.
\end{equation}
In other words, the indecomposable capped endo-permutation $kD$-modules are, up 
to isomorphism, precisely the modules of the form
\[
W_D(\alpha_0,\ldots,\alpha_{\elll-1})
:=\Omega_{D/D_0}^{\alpha_0}\circ\Omega_{D/D_1}^{\alpha_1}\circ\cdots
\circ\Omega_{D/D_{\elll-1}}^{\alpha_{\elll-1}}(k)
\]
with  $\alpha_i\in\{0,1\}$ for each $0\leq i\leq \elll-1$. 
\par
Furthermore, each indecomposable capped endo-permutation $kD$-module $W$ lifts 
in a unique way to an endo-permutation $\cO D$-lattice~$\widetilde{W}$ such that 
the determinant of the underlying representation is the trivial character. It is 
then said that $\widetilde{W}$ has \emph{determinant one}. Let~$\rho_{W}$ denote 
the $K$-character afforded by~$\widetilde{W}$. Then, our assumptions 
on $(K,\cO,k)$ ensure that for every $u\in P$, we have 
$\rho_{W}(u)\in\mathbb{Z}\setminus\{0\}$ and we define 
\[
 \omega_{W}(u):=
 \begin{cases}
 +1 \,\text{ if }\rho_{W}(u)>0\,,\\
 -1\,\text{ if }\rho_{W}(u)<0\,.
 \end{cases}
\]
See \cite[\S28 and \S52]{ThevenazBook} for details.  
 
\begin{lem}\label{lem:Omegaomega}%
Let $[W]\in\dade(D)$ and set $L:=\Omega_{D/D_{0}}(W)$. Then 
$\omega_{L}(u)=-\omega_{W}(u)$ for every $u\in D\setminus\{1\}$.
\end{lem}
\begin{proof}
Recall that $\Omega_{D/D_{0}}=\Omega$ is simply the Heller operator. Since $D$ 
is cyclic, the projective cover of $W$ is isomorphic to~$kD$ and we have a short 
exact sequence
\[
0\lra L\lra kD\lra W\lra 0\,.
\]
Then, $\cO D$ is a projective cover of $\widetilde{W}$. Now, as  $p$ is odd, any 
permutation $\cO P$-lattice has determinant one 
(see e.g.~\cite[Lemma~3.3(a)]{LT19}). Therefore, the kernel of the projective 
cover of $\widetilde{W}$ also has determinant one, and so we have a short exact 
sequence
\[
0\lra \widetilde{L}\lra \cO D\lra \widetilde{W}\lra 0\,.
\]
As projective characters vanish at $p$-singular elements, it follows that 
$\rho_{W}(u)+\rho_{L}(u)=0$ for every $u\in D\setminus\{1\}$. The claim follows.
\end{proof}

\begin{lem}\label{lem:bijectionsigns}%
For each $1\leq i\leq \elll$ let $u_{i}$ be a generator of the subgroup~$D_{i}$ 
of~$D$. Then, the map
\begin{center}
\begin{tabular}{cccl}
       $\Psi_{\elll}$\,:            &   $\mathbf{D}_{k}(D)$      & $\lra$ &    $ \{\pm1\}^{\elll}$     \\
                         &   $[W]$      & $\mapsto$     &   $(\omega_{W}(u_{1}),\ldots,\omega_{W}(u_{\elll}))$   
\end{tabular}
\end{center}
is bijective and  independent of the choice of the generators 
$u_{1},\ldots,u_{\elll}$.  Furthermore, if $W=W_D(\alpha_0,\ldots,\alpha_{\elll-1})$ and 
$1\leq i\leq \elll$, then 
\[
\omega_{W}(u_{i})
=
\begin{cases}
+1 \,\text{ if } \sum_{j=0}^{i-1}\alpha_{j}\equiv 0\pmod{2},  \\
-1 \,\text{ if }  \sum_{j=0}^{i-1}\alpha_{j} \equiv 1\pmod{2}.
\end{cases}
\]
\end{lem}
\begin{proof}
We proceed by induction on $\elll$. So first assume that $\elll=1$. Then 
$D\cong C_{p}$ and (\ref{eq:dadegroupcyclic}) yields 
$\dade(D)=\{[k],[\Omega_{D/D_{0}}(k)]\}$, where $k=W_{D}(\alpha_{0})$ with $\alpha_{0}=0$ 
and $\Omega_{D/D_{0}}(k)=W_{D}(\alpha_{0})$ with $\alpha_{0}=1$. Clearly,  
$\rho_{k}=1_{D}$ is the trivial character and so $\omega_{k}(u)=1$ for each 
$u\in D\setminus\{1\}$, whereas for $L:=\Omega_{D/D_{0}}(k)$ we have  
$\omega_{L}(u)=-1$  for each $u\in D\setminus\{1\}$ by 
Lemma~\ref{lem:Omegaomega}. Hence $\Psi_{\elll}([k])=(+1)$, 
$\Psi_{\elll}([\Omega_{D/D_{0}}(k)])=(-1)$ and all assertions hold in this case.
\par
Next, assume that $\elll>1$ and assume that  the assertions hold for $\elll-1$.  
For $u\in D$, write $\overline{u}:=uD_{1}\in D/D_{1}\cong C_{p^{\elll-1}}$, so 
that $\overline{u_{i}}$ is a generator of the cyclic subgroup of order 
$p^{i-1}$ of $D/D_{1}$ for each $2\leq i\leq \elll$.
\par
According to \cite[Exercise (28.3)]{ThevenazBook}, we can decompose $\dade(D)$ 
into the disjoint union 
\[
\dade(D)= \Inf_{D/D_{1}}^{D}\big(\dade(D/D_{1})\big) \sqcup 
\Omega_{D/D_{0}}\Big(\Inf_{D/D_{1}}^{D}\big(\dade(D/D_{1})\big)\Big).
\]
In other words, if $[W]\in \dade(D)$, then there exist $[V]\in\dade(D/D_{1})$  
and $\alpha_{0}\in\{0,1\}$ such that 
\[
W\cong \Omega_{D/D_{0}}^{\alpha_{0}}\big(\Inf_{D/D_{1}}^{D}(V)\big)\,.
\]
Now, on the one hand, if $\alpha_{0}=0$, then $D_{1}$ acts trivially on $W$. Thus, 
for every $u\in D$ we have $\rho_{W}(u)=\rho_{V}(\overline{u})$, implying that 
$\omega_{W}(u)=\omega_{V}(\overline{u})$. In particular, for every 
$u\in D_{1}\setminus\{1\}$ we have $\rho_{W}(u)=\dim_{k}(V)$ and hence 
$\omega_{W}(u)=+1$. On the other hand, if $\alpha_{0}=1$, then 
$\omega_{W}(u)= -\omega_{V}(\overline{u})$  for every 
$u \in D \setminus \{ 1 \}$ by Lemma~\ref{lem:Omegaomega}. In particular, 
$\omega_{W}(u)=-1$ for every $u \in D_{1} \setminus \{ 1 \}$.  So, in all cases, 
$\omega_{W}(u_{1})$ is independent of the choice of $u_{1}$. 
It follows that we can identify $\Psi_{\elll}$ with the map
\begin{center}
\begin{tabular}{cccl}
              &   $\dade(D)$      & $\lra$ &    $ \{\pm1\}\times \{\pm1\}^{\elll-1} $  \\
                                   & [W] & $\mapsto$  & $\big(\omega_{W}(u_{1}),\Psi_{\elll-1}([V])\big)$
\end{tabular}
\end{center}
where 
\[
\Psi_{\elll-1}: \dade(D/D_{1})\lra  \{\pm1\}^{\elll-1}, [V] \mapsto (\omega_{V}(\overline{u_{2}}),\ldots,\omega_{V}(\overline{u_{\elll}}))\,.
\]
By the induction hypothesis $\Psi_{\elll-1}$ is bijective and independent of 
the choice of the generators $u_{2},\ldots,u_{\elll}$. Therefore, so is 
$\Psi_{\elll}$ by the above argument.
\par
Finally, observe that if $W=W_D(\alpha_0,\ldots,\alpha_{\elll-1})$, then we have 
$V=W_{D/D_{1}}(\alpha_{1},\dots,\alpha_{\elll-1})$. It was proved above that 
$\omega_{W}(u_{1})=+1$ if $\alpha_{0}=0$ and $\omega_{W}(u_{1})=-1$ if $\alpha_{0}=1$, 
whereas for each $2\leq i\leq \elll$ the induction hypothesis and 
Lemma~\ref{lem:Omegaomega} imply that 
\begin{equation*}
\begin{split}
    \omega_{W}(u_{i}) =(-1)^{\alpha_{0}}\cdot\omega_{V}(\overline{u_{i}})
     & =
\begin{cases}
(-1)^{\alpha_{0}}\cdot (+1) \,\text{ if } \sum_{j=1}^{i-1}\alpha_{j}\equiv 0\pmod{2},  \\
(-1)^{\alpha_{0}}\cdot (-1) \,\text{ if }  \sum_{j=1}^{i-1}\alpha_{j} \equiv 1\pmod{2}
\end{cases} \\
    & = 
     \begin{cases}
+1 \,\text{ if } \sum_{j=0}^{i-1}\alpha_{j}\equiv 0\pmod{2},  \\
-1 \,\text{ if }  \sum_{j=0}^{i-1}\alpha_{j} \equiv 1\pmod{2},
\end{cases}
\end{split}
\end{equation*}
as claimed.
\end{proof}

\begin{lem}\label{lem:valuescharc}
Set $W:=W(\bB)$. Let $\xi$ denote the non-exceptional character in 
$\Irr_{K}(\bc)$ and let $u\in D\setminus\{1\}$. Then, $\xi(u)$ is a non-zero 
integer and the following assertions hold:

{\rm(i)} $\xi(u)>0$  if and only if $\omega_{W}(u)=+1$;

{\rm(ii)} $\xi(u)<0$  if and only if $\omega_{W}(u)=-1$. \\
Moreover, $\xi$ is constant on $D_{i}\setminus D_{i-1}$ for every 
$1\leq i\leq\elll$.
\end{lem}
\begin{proof}
As $\bc$ is a nilpotent block, there is a Morita equivalence between~$\bc$ 
and~$kD$. This equivalence can be lifted to a Morita equivalence over~$\cO$, 
which, by~\cite[Theorem~(52.8)(a)]{ThevenazBook}, in turn induces a unique 
character bijection $\Irr_{K}(\bc)\lra \Irr_{K}(D)$ determining the generalized 
decomposition numbers of $\bc$. It then follows 
from~\cite[Theorem~(52.8)(a)]{ThevenazBook} again and Dade's Theorem on cyclic 
blocks (see e.g. \cite[Theorem~68.1(8)]{DornhoffB}) that this bijection 
maps $\xi$ to the trivial character $1_{D}$, because $1_{D}$ is the unique 
$p$-rational element of $\Irr_{K}(D)$.
Therefore, Brauer's second main theorem (see~\cite[Theorem (43.4)]{ThevenazBook}) together 
with~\cite[Theorem~(52.8)(a)]{ThevenazBook} imply that
\begin{equation}\label{eq:xi(u)}
\xi(u)=\omega_{W}(u) \cdot m\,,
\end{equation}
where~$m$ is a sum of degrees of Brauer characters, hence a positive integer. 
Hence $\xi(u)$ is a non-zero integer and satisfies (i) and (ii). The last claim 
also follows from (\ref{eq:xi(u)}) as $\omega_{W}$ is constant on 
$D_{i}\setminus D_{i-1}$ for every $1\leq i\leq\elll$ by 
Lemma~\ref{lem:bijectionsigns}.
\end{proof}

We can now prove that the module $W(\bB)$ can be read off from the values of  
the non-exceptional characters in $\Irr_{K}(\bB)$ at the non-trivial elements 
of $D$.
\par
In the proof of this result, we will need to consider some specific modules,
called \emph{hooks}. We refer to our paper~\cite{HL19} for details, 
but these modules can be thought of as being the indecomposable $\bB$-modules 
lying on the rim of the stable Auslander-Reiten quiver of the block.

\begin{thm}\label{thm:readingWfromIrr(B)}
Let $p$ be an odd prime. Let $\bB$ be a $p$-block of $kG$ with a cyclic defect 
group $D\cong C_{p^\elll}$ where~$\elll\geq 1$. Set $W:=W(\bB)$ and let~$\chi$ 
be a non-exceptional character in $\Irr_{K}(\bB)$. For each $1\leq i\leq \elll$, 
let $u_{i}\in D_{i}\setminus D_{i-1}$. Then, $\chi(u_{i})$ is a non-zero integer 
for every $1\leq i\leq\elll$ and  the following assertions hold. 
\begin{enumerate}[label={\rm(\alph*)}]
\item If $\chi(u_{1})>0$, then for every $2\leq i\leq\elll$, we have

{\rm(i)} $\chi(u_{i})>0$ if and only if  $\omega_{W}(u_{i})=+1$;  

{\rm(ii)} $\chi(u_{i})<0$ if and only if  $\omega_{W}(u_{i})=-1$; 
\item If $\chi(u_{1})<0$, then for every $2\leq i\leq\elll$, we have

{\rm(i)}  $\chi(u_{i})>0$ if and only if  $\omega_{W}(u_{i})=-1$;  

{\rm(ii)} $\chi(u_{i})<0$ if and only if  $\omega_{W}(u_{i})=+1$.
\end{enumerate}
Moreover, $W$ is determined up to isomorphism by the character values 
$\chi(u_{1}),\ldots,\chi(u_{\elll})$.
\end{thm}

\begin{proof}
First assume that $G=C_{G}(D_{1})$, so that $\bB=\bc$. Let $\xi$ denote the 
non-exceptional character in $\Irr_{K}(\bc)$.  By Lemma~\ref{lem:valuescharc}, 
clearly, $\xi(u_{i})$ is a non-zero integer for every $1\leq i\leq\elll$ and 
$\xi(u_{1})>0$ as $D_{1}$ acts trivially on $W$ by Clifford theory. Hence 
only (a) can happen in this case and (i) and (ii) are given by 
Lemma~\ref{lem:valuescharc}.
\par
Next assume that  $G=N_1$, so that $\bB=\bb$.
If $\psi$ is a non-exceptional $K$-character of $\bb$, then $\psi$ lies 
above~$\xi$ and 
\[
\Res^{N_1}_{C_{G}(D_{1})}(\psi)=\sum_{i=1}^{t}{^{g_{i}}\!\xi}
\]
for suitable elements $g_{1},\ldots,g_{t}\in N_1$. Thus, for any 
$u_{i}\in D_{i}\setminus D_{i-1}$, we have
\[
\psi(u_{i})=\sum_{j=1}^{t}\xi(g_{j}^{-1}u_{i}g_{j})\,.
\] 
Now, if $g_{j}^{-1}u_{i}g_{j}$ is not conjugate  to an element of~$D$ in 
$C_{G}(D_{1})$, then $\xi(g_{j}^{-1}u_{i}g_{j})=0$ by Green's theorem on zeroes 
of characters; see~\cite[(19.27)]{CR1}. If $g_{j}^{-1}u_{i}g_{j}$ is conjugate  
to an element  $u$ of $D$ in $C_{G}(D_{1})$, then $|u|=|u_{i}|$, and so 
$\xi(g_{j}^{-1}u_{i}g_{j})=\xi(u_{i})$ by Lemma~\ref{lem:valuescharc}.  
Therefore, Lemma~\ref{lem:valuescharc} also implies that $\psi(u_{i})$ is a 
non-zero integer as $\xi(u_{i})$ is and we have 
\begin{equation}\label{eq:psi>0}
\psi(u_{i})>0\quad \text{if and only if} \quad \xi(u_{i})>0, 
\end{equation}
and
\begin{equation}\label{eq:psi<0}
\psi(u_{i})<0\quad \text{if and only if} \quad \xi(u_{i})<0.
\end{equation}
Again in this case only (a) can happen and (i) and (ii) 
follow from (a) and (b) of Lemma~\ref{lem:valuescharc}.
\par
{
Assume now that $G$ is arbitrary. 
Then $\chi$ is the $K$-character of a lift~$\widehat{H}$ to $\cO$ 
of a hook~$H$ of $\bB$; see~\cite[\S3.5 and \S4.2]{HL19}.
(In fact, if the inertial index of $\bB$ is at least $2$, then all lifts of $H$ afford $\chi$.)
Let $f(H)$ and $f(\widehat{H})$ denote the Green correspondents 
in $N_{1}$ of $H$ and $\widehat{H}$, respectively.  
Then, $f(H)$ belongs to $\bb$ and is also a hook of $\bb$.
By \cite[Lemma~4.1]{HL19}, we have that $f(H)$ is simple if $\chi(u_{1})>0$, 
whereas $f(H)$ is the Heller-translate of a simple $\bb$-module if $\chi(u_{1})<0$.
\par
So, first assume that $\chi(u_{1})>0$. Then, by \cite[Lemma~4.2 and 
Corollary~4.3]{HL19}, $f(\widehat{H})$ lifts $f(H)$ and  affords a 
non-exceptional $K$-character $\psi$ of $\bb$ satisfying 
$\chi(u_{i})=\psi(u_{i})$ for every ${u_{i}\in D_{i}\setminus D_{i-1}}$ and every 
$1\leq i\leq\elll$.  Thus (i) and (ii) of (a) follow from~(\ref{eq:psi>0}) 
and~(\ref{eq:psi<0}) and Lemma~\ref{lem:valuescharc}. 
\par
Next, assume that $\chi(u_{1})<0$. Then, an analogous argument yields that
a lift of $\Omega(f(H)) \cong f(\Omega(H))$ affords a non-exceptional 
$K$-character $\psi$ of $\bb$ satisfying 
$\chi(u_{i})=-\psi(u_{i})$ for every $u_{i}\in D_{i}\setminus D_{i-1}$ and 
every $1\leq i\leq\elll$, since projective characters vanish at $p$-singular 
elements. Thus (i) and (ii) of (b) follow from (\ref{eq:psi>0}) 
and~(\ref{eq:psi<0}) and Lemma~\ref{lem:valuescharc}.
}
\par
Finally, by Lemma~\ref{lem:bijectionsigns}, up to isomorphism $W$ is uniquely 
determined by the signs $\omega_{W}(u_{1}),\ldots,\omega_{W}(u_{\elll})$, where 
as seen  above $\omega_{W}(u_{1})=+1$ by definition of $W$. Hence, the last 
claim is  immediate from (a) and~(b).
\end{proof}

As a consequence, we can express $W(\bB)$ as 
$W(\bB)=W_D(\alpha_0,\ldots,\alpha_{\elll-1})$ where the integers 
$\alpha_0,\ldots,\alpha_{\elll-1}\in\{0,1\}$ are determined recursively using the 
formulae of Lemma~\ref{lem:bijectionsigns}.

\begin{cor}\label{cor:W(B)=k}
For each $1\leq i\leq \elll$, let $u_{i}\in D_{i}\setminus D_{i-1}$. Then, 
$W(\bB)$ is trivial if and only if there exists a non-exceptional character 
$\chi$ in $\Irr_{K}(\bB)$ such that either $\chi(u_{i})>0$ for all 
$1\leq i\leq\elll$ or $\chi(u_{i})<0$ for all $1\leq i\leq\elll$. 
\end{cor}

\begin{proof}
The trivial $kD$-module affords the trivial character of~$D$, implying that 
$\omega_{k}(u_{i})=+1$ for every $1\leq i\leq \elll$. Hence the claim 
follows from the bijection $\Psi_{\elll}$ in Lemma~\ref{lem:bijectionsigns} and 
Theorem~\ref{thm:readingWfromIrr(B)}.
\end{proof}

\subsection{Criteria for the triviality of~$W(\bB)$}
Here,~$p$ is an arbitrary prime. We provide further theoretical criteria for the 
triviality of $W(\bB)$ not involving character theory.

\begin{lem}\label{lem:B_0}\label{NG(D_1)}
In each of the following cases $W(\bB) \cong k$.

{\rm (a)}  If $\bB$ is the principal block of $kG$; 

{\rm (b)}  if $C_G(D)=C_G(D_1)$ or $N_G(D)=N_G(D_1)$;

{\rm (c)}  if $p = 2$ and $|D| = 4$.
\end{lem}
\begin{proof}
(a) If $\bB$ is the principal block of $kG$, then its Brauer correspondent $\bb$ 
in $N_1$ is the principal block of $kN_1$. Hence the trivial $kN_1$-module is a 
simple $\bb$-module. It follows that $W(\bB)$ is trivial.

(b) Since $D$ is normal in $C_G(D)$ and in $N_G(D)$ it follows from Clifford 
theory that $D$ acts trivially on the unique simple $kC_G(D)$-module and on the 
simple $kN_G(D)$-modules, hence they are trivial source modules. As we assume 
that $C_G(D)=C_G(D_1)$ or $N_G(D)=N_G(D_1)$, it follows that $W(\bB)$ is 
trivial.

(c) There are only two indecomposable capped endo-permutation modules of $kD$, 
namely~$k$ and $\Omega(k)$, where $\Omega$ denotes the Heller operator. 
But~$D_1$ does not act trivially on $\Omega(k)$, hence $W(\bB) \cong k$.
\end{proof}

Notice that Lemma~\ref{NG(D_1)}(b) applies in particular if $D \leq Z(G)$ or 
$D = D_1$.

\section{Reduction theorems}\label{sec:red}

This section is devoted to Clifford theory with respect to the source algebra 
equivalence classes of blocks with cyclic defect groups. 
{Throughout this 
section,~$G$ is a finite group and~$\mathbf{B}$ a $p$-block of~$G$ with a 
non-trivial cyclic defect group~$D$. For further notation refer to 
Subsection~\ref{ssec:cyclicBlocks}.
Our aim is in particular to extend the results of \cite[Section~$4$]{Fei84}. 
When possible, we state these results in terms of source algebra equivalences, 
which obviously implies that the module $W(\bB)$ is left unchanged.
Else we study the behavior of  $W(\bB)$ under each of these results. 
}

By definition, the kernel of~$\mathbf{B}$ is the intersection of the kernels of
the characters in $\Irr( \mathbf{B} )$.

We begin with \cite[Lemma~$4.1$]{Fei84}, which is a well-known result on inflation.

\begin{lem}
\label{Inflation}
Let $H \unlhd G$.
Suppose that~$H$ is in the kernel of~$\mathbf{B}$ and that $p \nmid |H|$. Put 
$\bar{G} := G/H$ and write $\bar{\ } : kG \rightarrow k\bar{G}$ for the 
canonical epimorphism. Then~$\mathbf{B}$ and its image $\bar{\mathbf{B}}$ 
in $k\bar{G}$ are source algebra equivalent.
\end{lem}
\begin{proof}
As $p \nmid |H|$, the image~$\bar{\mathbf{B}}$ of~$\mathbf{B}$ is isomorphic
to~$\mathbf{B}$. Moreover, 
the canonical epimorphism is an isomorphism of interior $D$-algebras between
$\mathbf{B}$ and $\bar{\mathbf{B}}$. This gives our claim.
\end{proof}

Then \cite[Lemma~$4.2$]{Fei84} is the first Fong reduction, 
which can also be stated in terms of source-algebra equivalences.

\begin{lem}
\label{FongReynolds}
Let $H \unlhd G$ and let~$\mathbf{B}_0$ be a $p$-block of~$H$. Then there is a
bijection between the $p$-blocks of~$G$ covering $\bB_0$ and the 
$p$-blocks of $\Stab_G( \mathbf{B}_0 )$ covering~$\mathbf{B}_0$, under which
corresponding blocks are source algebra equivalent.
\end{lem}
\begin{proof}
See \cite[Theorem~$6.8.3$]{LinckBook}.
\end{proof}

The following lemma studies the behavior of $W( \mathbf{B} )$ in the situation
of \cite[Lemma~$4.3$ and Lemma~$4.4$]{Fei84}.

\begin{lem}
\label{Feit4.3and4.4}
Let $H \unlhd G$ be a normal subgroup and let~$\mathbf{B}_0$ be a $G$-stable 
$p$-block of~$H$ covered by~$\mathbf{B}$. Assume that~$D \cap H$ is a defect 
group of~$\mathbf{B}_0$. If $\{1\}\lneq D\cap H\leq D$, then 
$W( \mathbf{B}_0 )\cong  \Capp\left(\Res^{D}_{D \cap H}( W( \mathbf{B} ) ) \right)$ 
as $k[D\cap H]$-modules. In particular  $W( \mathbf{B}_0 )\cong k$ provided 
$W( \mathbf{B} )\cong  k$, and $W( \mathbf{B}_0 ) \cong W( \mathbf{B} )$ 
provided $D\cap H = D$. 
\end{lem}

\begin{proof}
Let~$\mathbf{b}_0$ denote the Brauer correspondent of~$\mathbf{B}_0$ 
in~$N_H( D_1 )$. By definition $W( \mathbf{B}_0 ) \cong W( \mathbf{b}_0 )$ 
and $W( \mathbf{B} ) \cong W( \mathbf{b} )$, hence it suffices to prove that 
$W( \mathbf{b}_0 ) \cong \Capp\left(\Res^{D}_{D \cap H}( W( \mathbf{b} ) ) \right)$. 
 \par
Since~$D$ is cyclic and $\{1\}\lneq D\cap H$,  we have~{$D_{1}\leq D\cap H\leq D$} 
and the following inclusions of subgroups:
\[
\begin{tikzcd}[row sep=0.1mm, column sep=small]
  G^{} \arrow[dd, dash]    \arrow[rd, dash, yshift=-0.2ex] &  & \\
  &  N_{G}(D_{1}) \arrow[dd, dash]   \arrow[rd, dash, yshift=-0.2ex]  & \\
  H^{}    \arrow[rd, dash, yshift=-0.2ex]   & & N_{G}(D \cap H) \arrow[dd, dash]   \\
  &  N_{H}(D_{1})  \arrow[rd, dash, yshift=-0.2ex] &  \\ 
  & &    N_{H}(D \cap H) \\
  & & &
\end{tikzcd}
\]
where $N_{H}(D_{1})=H \cap N_{G}(D_{1})\lhd N_{G}(D_{1})$. 
Therefore, by the generalized version of the Harris-Kn{\"o}rr correspondence 
(see \cite[Theorem]{HK85}) given in \cite[Theorem~$6.9.3$]{LinckBook}, there is 
a bijection
$$\mbox{Bl}_{p}(N_{G}(D_{1}) \mid \mathbf{b}_0 ) 
\overset{\sim}{\longleftrightarrow} \mbox{Bl}_{p}(G \mid \mathbf{B}_0)$$
given by the Brauer correspondence. Thus~$\mathbf{b}$ covers~$\mathbf{b}_0$. 
Moreover, it follows from the uniqueness of the Brauer correspondent 
that~$\mathbf{b}_{0}$ is $N_{G}(D_{1})$-stable (else~$\bB_0$ would not be 
$G$-stable).
\par
Now, let $M$ be a simple $\mathbf{b}$-module with $kD$-source $S$. By definition 
$W( \mathbf{b} )\cong S$, and  $M\mid \Ind_{D}^{N_{G}(D_{1})}(S)$. By Clifford 
theory,  any indecomposable direct summand 
$M_{0}$ of $\Res^{N_{G}(D_{1})}_{N_{H}(D_{1})}(M)$ is a simple 
$\mathbf{b}_{0}$-module. Moreover, 
\[
M_{0} \mid \Res^{N_{G}(D_{1})}_{N_{H}(D_{1})}(M)  \mid 
\Res^{N_{G}(D_{1})}_{N_{H}(D_{1})} \Ind_{D}^{N_{G}(D_{1})}(S),
\]
and the Mackey formula yields
\begin{equation}
\label{MackeyDec}
M_{0} \,\bigm|\, \bigoplus_{x\in[N_{H}(D_{1})\backslash N_{G}(D_{1}) / D]} 
\Ind_{^{x}\!D\cap N_{H}(D_{1})}^{N_{H}(D_{1})}
\Res^{^{x}\!D}_{^{x}\!D\cap N_{H}(D_{1})}(^{x}\!S)\,.
\end{equation}
Since~$N_{H}(D_{1}) \lhd N_{G}(D_{1})$, for any $x\in N_{G}(D_{1})$ we have 
\[
^{x}\!D\cap N_{H}(D_{1})\cong D\cap\,^{x^{-1}}\!N_{H}(D_{1})=D\cap N_{H}(D_{1})=D\cap H\,.
\]
\par
Moreover, as~$S$ is a capped endo-permutation $kD$-module, it follows that 
$\Res^{D}_{D\cap N_{H}(D_{1})}(S) =
\Res^{D}_{D\cap H}(S)\cong (S_{0})^{\oplus m}\oplus X_{0}$, 
where~$m$~is a positive integer,
$S_{0}:=\Capp\left( \Res^{D}_{D \cap H}( S ) \right)$ and all the indecomposable 
direct summands of $X_{0}$ have a vertex strictly contained in 
$D\cap N_{H}(D_{1})$. Then, as $M_{0}$ has vertex~$D \cap H$, it follows 
from~(\ref{MackeyDec}) that there exists $x\in N_{G}(D_{1})$ such that
$$ M_{0} \mid  \Ind_{^{x}\!D\cap N_{H}(D_{1})}^{N_{H}(D_{1})} ( ^{x}\!S_{0} )\,.$$
Therefore $(^{x}\!D\cap N_{H}(D_{1}), ^{x}\!S_{0} )$ is a vertex-source pair 
for~$M_{0}$. Thus, identifying $^{x}\!D\cap N_{H}(D_{1})$ with $D\cap H$, we 
obtain that $W( \mathbf{b}_{0} )\cong S_{0}$ as $k[D \cap H]$-modules, because 
the isomorphism class of an indecomposable $k[D\cap H]$-module is entirely 
determined by its $k$-dimension. The claim follows. 
\end{proof}

\begin{cor}
\label{Feit4.3Cor}
Let $H \unlhd G$ be a normal subgroup and let~$\mathbf{B}_0$ be a $p$-block 
of~$H$ covered by~$\mathbf{B}$. If $D \leq H$, then $W( \bB_0 ) \cong W( \bB )$.
\end{cor}
\begin{proof}
This is immediate from Lemmas~\ref{FongReynolds} and \ref{Feit4.3and4.4}.
\end{proof}

The above corollary supplements \cite[Lemma~$4.3$]{Fei84}. Let us 
introduce some notation to make this more precise.
Recall that $\sigma( \mathbf{B} )$ denotes the Brauer tree of~$\mathbf{B}$.

\begin{defn}
\label{StronglySimilar}
{\rm 
We call two cyclic $p$-blocks~$\mathbf{B}'$ and~$\mathbf{B}''$ 
\textit{strongly similar}, if they have the same defect, if their non-embedded
Brauer trees are similar in 
the sense of \cite[Section~$2$]{Fei84}, and if $W( \mathbf{B}' ) \cong 
W( \mathbf{B}'' )$, when the defect groups of~$\mathbf{B}'$ and~$\mathbf{B}''$ 
are identified.
}
\end{defn}

It is clear that source algebra equivalent blocks are strongly similar. Assume 
now the hypothesis of Corollary~\ref{Feit4.3Cor}. Then \cite[Lemma~$4.3$]{Fei84} 
and Corollary~\ref{Feit4.3Cor} imply that~$\mathbf{B}$ and~$\mathbf{B}_0$ are 
strongly similar. There is a situation in which these two blocks are even 
source algebra equivalent.

\begin{lem}
\label{Feit4.3Extended}
Assume the hypothesis of {\rm Lemma~\ref{Feit4.3and4.4}} and also that 
${D \leq H}$. Put
$$G[\mathbf{B}_0] := \{ g \in G \mid g \text{\ acts as inner automorphisms 
on\ } \mathbf{B}_0 \}.$$ 
Suppose that $G[\mathbf{B}_0] = G$. Then $\mathbf{B}$ and $\mathbf{B}_0$ are 
source algebra equivalent.
\end{lem}
\begin{proof}
This follows from \cite[Proposition~$5$ and Theorem~$7$]{K90}; see also 
\cite[Proposition~$2.2$]{KKL12}.
\end{proof}

The following extends Linckelmann's version \cite[Theorem~$6.8.13$]{LinckBook} 
of Fong's second reduction in the context of cyclic blocks.
{Important to us is the fact that this result describes a Morita equivalence
induced by a bimodule with endo-permutation source. 
We give a proof which makes this endo-permutation module explicit.
}

\begin{prop}
\label{Fong2ndReduction}
Let $H \unlhd G$ and let~$\mathbf{B}_0$ be a $p$-block of~$H$ covered
by~$\mathbf{B}$. Suppose that~$\mathbf{B}_0$ is $G$-stable and of 
defect~$0$. Let~$V$ be the simple $\mathbf{B}_0$-module.

Then there is a natural action of~$D$ on~$V$, which gives~$V$ the structure of a
capped endo-permutation $kD$-module. Moreover, there is a central extension 
$\widehat{G/H}$ of $G/H$ with a center of order prime to~$p$, and a 
$p$-block~$\hat{\bB}$ of~$\widehat{G/H}$ with defect group~$\hat{D}$, such 
that~$\bB$ and~$\hat{\bB}$ are basic Morita equivalent. In particular,~$D$
and~$\hat{D}$ are isomorphic. 

If~$G/H$ is perfect,~$\widehat{G/H}$ can be chosen to be perfect as well.

Let~$M$ and $\hat{M}$ denote a simple $\bB$-module and a simple
$\hat{\bB}$-module corresponding under this Morita equivalence. Identify~$D$ 
with~$\hat{D}$ and let $W(M)$ respectively $W( \hat{M} )$ denote the $D$-sources
of~$M$ respectively~$\hat{M}$. Then $[W( M )] = [\Capp(V)] + [W( \hat{M} )]$ 
in~$\dade(D)$,
\end{prop}
\begin{proof}
As~$\mathbf{B}_0$ has defect~$0$, first observe that $D \cap H = \{ 1 \}$, so 
that we may identify~$D$ with a subgroup of $G/H$. 
Also $\mathbf{B}_0 \cong \End_{k}(V)$ as $k$-algebra. 
Secondly, as $H\unlhd G$, we have that $\bB_0$ is a $kD$-module 
under the conjugation action of $D$, and as such is a direct 
summand of~$kH$, hence  a permutation $kD$-module.
Moreover, the action of~$H$ on~$V$ extends in a unique way to an
action of the semidirect product~$DH$; see \cite[III, Corollary~$3.16$]{Feit}.
Thus, by definition,~$V$ is an endo-permutation $kD$-module.

We claim that the restriction of~$V$ to~$D$ contains an indecomposable direct
summand~$V_0$ with vertex~$D$. Indeed, let 
$\iota \in Z( kH )$ denote the primitive idempotent with 
$\bB_0 = \iota kH \iota$. As~$D$ is a defect group of 
the block~$\bB$ which covers~$\bB_0$, we have $\Br_D( \iota ) \neq 0$. 
As~$\bB_0$ has a $k$-basis which is permuted by~$D$, the image $\Br_D( \iota )$ 
is spanned by the $D$-fixed points of this basis. Thus $\bB_0 \cong \End_k( V )$ 
has a trivial direct summand as a $kD$-module. It follows from 
\cite[Theorem~$3.1.9$]{BensonBookI}, that $V$, as a $kD$-module, has an
indecomposable direct summand~$V_0$ of dimension prime to~$p$. In 
particular,~$V_0$ has vertex~$D$ and thus~$V$ is a capped endo-permutation 
$kD$-module.

By the Skolem--Noether
theorem, any $k$-algebra automorphism of $\End_k(V)$ is inner. Thus every
element $x \in G$ yields a unit $s_x \in \bB_0^\times$ such that 
$x t x^{-1} = s_x t s_x^{-1}$ for all $t \in \bB_0$. Moreover, for any 
$x, y \in H$, there is $\alpha(x,y) \in k^\times$ such that $s_{xy} = 
\alpha(x,y) s_xs_y$ with $\alpha(x,y) = 1$ whenever $x, y \in H$.
We thus obtain an element ${\alpha} \in Z^2( G/N, k^\times )$. We may
assume that the set of values of~$\alpha$ lies in a finite field; see,
e.g.\ \cite[Lemma~$11.38$]{CR1}.

Let~$k_\alpha[G/H]$ denote the twisted group algebra of~$G/H$ with respect 
to~$\alpha$, and let~$\widetilde{G/H}$ be the finite central extension of~$G/H$ 
corresponding to $\alpha \in Z^2( G/H, k^\times )$. If~$G/H$ is perfect, put
$\widehat{G/H} := [\widetilde{G/H},\widetilde{G/H}]$, the commutator subgroup 
of~$\widetilde{G/H}$. Otherwise, let $\widehat{G/H} := \widetilde{G/H}$.
Then~$\widehat{G/H}$ is a central extension of~$G/H$, which is perfect if~$G/H$ 
is, and 
$k[\widehat{G/H}]\epsilon \cong k_\alpha[G/H]$ as $k$-algebras for some central 
idempotent $\epsilon \in k[\widehat{G/H}]$; see 
\cite[Proposition~$10.5$]{ThevenazBook} and \cite[Theorem~$3.4$]{Sal}.

There is an isomorphism 
\begin{equation}
\label{CrucialIsomorphism0}
kG \otimes_{kH} \bB_0 \rightarrow \bB_0 \otimes_k k_\alpha[G/H];
\end{equation}
see \cite[Theorem~$6.8.13$]{LinckBook}. In fact,~(\ref{CrucialIsomorphism0}) is
an isomorphism of interior $D$-algebras, where the $D$-algebra structure on 
$k_\alpha[G/H]$ arises from the embedding of $D$ into $G/H$; see 
\cite[Proposition~$3.3$]{Sal}. We thus obtain an isomorphism 
\begin{equation}
\label{CrucialIsomorphism1}
\Phi: kG \otimes_{kH} \bB_0 \rightarrow \bB_0 \otimes_k k[\widehat{G/H}]\epsilon
\end{equation}
of interior $D$-algebras, where the  $D$-algebra structure on 
$k[\widehat{G/H}]\epsilon$ arises from an embedding of~$D$ into~$\widehat{G/H}$. 
Let~$\hat{D} \leq \widehat{G/H}$ denote the image of this embedding.
Now~$\bB$ is a block of $kG \otimes_{kH} \bB_0$, and 
hence there is a block $\widehat{\bB} = \widehat{\bB}\epsilon$ of $k[\widehat{G/H}]$, 
such that 
\begin{equation}
\label{CrucialIsomorphism2}
\Phi_{\bB} : \bB \rightarrow \bB_0 \otimes_k \widehat{\bB}
\end{equation}
is an isomorphism of interior $D$-algebras. 

Identifying $\bB$ and $\bB_0 \otimes_k \widehat{\bB}$ 
via~(\ref{CrucialIsomorphism2}), we obtain a Morita equivalence 
$\mbox{\rm mod-}\widehat{\bB} \rightarrow \mbox{\rm mod-}\bB$, which is
basic by \cite[Subsection~$7.1$ and Corollary~$7.4$]{PuigBook}. This implies 
in particular, that
the image~$\hat{D}$ of~$D$ under~$\Phi_{\bB}$ is a defect group of~$\hat{\bB}$.
Under this Morita equivalence,~$\hat{M}$
is sent to $V \otimes_k \hat{M}$. Thus, as~(\ref{CrucialIsomorphism2}) is an
isomorphism of interior $D$-algebras, this yields our claim about the sources
of~$M$ and $\hat{M}$.
\end{proof}

\section{Preliminaries on groups of Lie type}

In order to investigate Question~\ref{QuestionA}(c), we need to introduce
some concepts and notation from the theory of finite groups of Lie type,
where we largely follow the book \cite{GeMa} of Geck and Malle.
Let~$r$ be a prime. Let $\mathbf{G}$ denote a connected reductive 
algebraic group over $\bar{\mathbb{F}}_r$, and let~$F$ be a Steinberg 
morphism of~$\mathbf{G}$. If~$\mathbf{H}$ is a closed subgroup of~$\mathbf{G}$,
we write $\mathbf{H}^\circ$ for its connected component, and if~$\mathbf{H}$ is
also $F$-stable, we write $H := \mathbf{H}^F$ for the finite
group of $F$-fixed points of~$\mathbf{H}$. By~$\mathbf{G}^*$ we denote a 
connected reductive algebraic group dual to~$\mathbf{G}^*$, equipped with a 
dual Steinberg morphism, which is also denoted by~$F$. An $F$-stable Levi
subgroup of~$\bG$ is called a \textit{regular subgroup of}~$\bG$.

We record the following factorization lemma.

\begin{lem}
\label{FactorizationLemma}
Let $\mathbf{G}$ denote a connected reductive algebraic group over
$\bar{\mathbb{F}}_r$, and let~$F$ be a Steinberg morphism of~$\mathbf{G}$.

Then $\mathbf{G} = \mathbf{Z}\mathbf{H}$, with 
$\mathbf{Z} = Z( \mathbf{G} )^\circ$ and $\mathbf{H} = [\mathbf{G},\mathbf{G}]$.
Moreover, $|G| = |Z| \cdot |H|$. In particular, $ZH \unlhd G$ with 
$[G\colon\!ZH] = |Z \cap H|$, and $|Z \cap H|$ divides $|Z( \mathbf{H} )|$.
\end{lem}
\begin{proof}
The factorization $\mathbf{G} = \mathbf{Z}\mathbf{H}$ is standard.
The multiplication map $\mathbf{Z} \times \mathbf{H} \rightarrow \mathbf{G}$ is 
a surjective, $F$-equivariant morphism of algebraic group with a finite kernel, 
and thus $|G| = |Z| \cdot |H|$; see \cite[Proposition~$1.4.13$]{GeMa}.
\end{proof}

\begin{cor}
\label{FactorizationCorollary}
Let the notation be as in {\rm Lemma~\ref{FactorizationLemma}}. Let~$p$ be a 
prime dividing~$|Z|$ but not~$|Z( \mathbf{H} )|$. Let~$D$ be a cyclic 
$p$-subgroup of~$G$ containing the Sylow $p$-subgroup~$P$ of~$Z$. Then $D = P$.
\end{cor}
\begin{proof}
Lemma~\ref{FactorizationLemma} implies that every Sylow $p$-subgroup of~$G$ is 
of the form $P \times Q$ for some Sylow $p$-subgroup~$Q$ of~$H$. As $P \leq D$,
this yields our claim.
\end{proof}

Choose an $F$-stable maximally split
torus~$\mathbf{T}$ of~$\mathbf{G}$ and consider the character group
$X := X( \mathbf{T} )$, a free abelian group of finite rank. The action of~$F$
on~$\mathbf{T}$ induces a linear map~$\varphi$ on~$\mathbb{R} \otimes_\mathbb{Z} X$,
which factors as $\varphi = q \varphi_0$ for a positive real number~$q$ and
a linear map $\varphi_0$ of finite order; see 
\cite[Proposition~$1.4.19$(b)]{GeMa}. Notice that~$q$ is the absolute value
of all eigenvalues of~$\varphi$, and thus~$q$ and~$\varphi_0$ are uniquely 
determined by the pair $(\mathbf{G},F)$. Moreover, $q^d = r^f$ for positive 
integers $d, f$; see \cite[Proposition~$1.4.19$]{GeMa}. If~$q$ is an integer, 
it is a power of~$r$. In \cite[Section~$1$A]{BrouMa}, Brou{\'e} and Malle define
the concept of a \textit{complete root datum}, also called a \textit{generic 
finite reductive group}; for a slightly more general definition see 
\cite[Definition~$1.6.10$]{GeMa}. Associated 
with the pair $(\mathbf{G},F)$ is a generic finite reductive group~$\mathbb{G}$; 
see \cite[Example~$1.6.11$]{GeMa}. Conversely, if~$q$ is an integer, 
$(\mathbf{G},F)$ can be constructed from~$\mathbb{G}$ as explained in 
\cite[Section~$2$]{BrouMa}. Attached to~$\mathbb{G}$ is a real polynomial, 
called the \textit{order polynomial of}~$\mathbb{G}$; for its definition see 
\cite[D\'{e}finition~$1.9$]{BrouMa} or \cite[Definition~$1.6.10$]{GeMa} and for
its relevance see \cite[Th{\'e}or{\`e}me~$2.2$]{BrouMa} or 
\cite[Remark~$1.6.15$]{GeMa}. By a slight abuse of language, the order 
polynomial of~$\mathbb{G}$ is also called the order polynomial of $(\bG,F)$.

We now cite the condition~{\rm ($\ast$)} formulated 
in~\cite[Subsection~$2.1$]{MalleAb}. Assume that the positive real number~$q$ 
arising from $(\mathbf{G},F)$ as above is an integer. Let~$p$ be a prime 
different from~$r$. Then~$p$ satisfies condition~{\rm ($\ast$)} with respect to 
$(\mathbf{G},F)$, if there is a unique integer~$d$ such that 
$p \mid \Phi_d(q)$ and~$\Phi_d$ divides the order polynomial of $(\bG,F)$. 
Here,~$\Phi_d$ denotes the $d$th cyclotomic polynomial over~$\mathbb{Q}$. If~$p$ 
satisfies~{\rm ($\ast$)}, then~$p$ is odd and good for~$\mathbf{G}$ and the Sylow
$p$-subgroup of~$G$ are abelian; see~\cite[Lemma~$2.1$ and Proposition~$2.2$]{MalleAb}.

\begin{lem}
\label{LargeEll}
Let $\mathbf{G}$ denote a connected reductive algebraic group over 
$\bar{\mathbb{F}}_r$, and let~$F$ be a Steinberg morphism of~$\mathbf{G}$. 
Assume that the positive real number~$q$ associated to $(\mathbf{G},F)$ as in 
\cite[Proposition~$1.4.19$]{GeMa} is an integer (which then is a power of~$r$).
Let $p$ be a prime different from~$r$ which satisfies condition~{\rm ($\ast$)} 
given in~\cite{MalleAb}. Assume in addition that~$p$ does not divide
$|Z(\mathbf{G})/Z(\mathbf{G})^\circ||Z(\mathbf{G}^*)/Z(\mathbf{G}^*)^\circ|$.

Let $D \leq G$ denote a non-trivial cyclic radical $p$-subgroup of~$G$ of 
order~$p^l$. Then $C_{\mathbf{G}}( D ) = C_{\mathbf{G}}( D_i )$ for every $1 \leq i \leq l$.
\end{lem}
\begin{proof}
It suffices to prove the claim for $i = 1$.
Put $\mathbf{L} := C_{\mathbf{G}}( D )$ and 
$\mathbf{L}_1 := C_{\mathbf{G}}( D_1 )$. Then~$\mathbf{L}$ and~$\mathbf{L}_1$ 
are regular subgroups of~$\mathbf{G}$ as $p$ is good for~$\mathbf{G}$ and 
does not divide $|Z(\mathbf{G}^*)/Z(\mathbf{G}^*)^\circ|$; see 
\cite[Corollary $2.6$]{GeHi1}. Clearly,  
$\mathbf{L} \leq \mathbf{L}_1$ and $Z( \mathbf{L}_1 ) \leq Z( \mathbf{L} )$ as
$D \leq \mathbf{L}_1$. Also, $D = O_{p}( Z( L ) )$ by 
Lemma~\ref{AbelianRadicalSubgroup}.

There is a surjection $Z( \mathbf{G} )/Z(\mathbf{G})^\circ \rightarrow 
Z( \mathbf{L} )/Z( \mathbf{L} )^\circ$; see \cite[Pro\-position~$4.2$]{CeBo2}. 
By assumption, $p \nmid |Z( \mathbf{G} )/Z( \mathbf{G} )^\circ|$, and hence 
$D \leq (Z( \mathbf{L} )^\circ)^F$. Analogously, 
$D_1 \leq (Z( \mathbf{L}_1 )^\circ)^F$.

Let~$d$ denote the order of~$q$ modulo~$p$, and let $\mathbf{T}_d$ denote the 
Sylow $\Phi_d$-torus of $Z( \mathbf{L} )^\circ$; for the definition of a Sylow 
$\Phi_d$-torus see \cite[p.~$254$]{BrouMa}. Since $O_{p}( Z( L ) ) = D$ is 
cyclic, the order of $\mathbf{T}_d^F$ equals $\Phi_d(q)$ and the rank 
of~$\mathbf{T}_d$ equals $\deg( \Phi_d )$; see \cite[Proposition~$3.3$]{BrouMa}.

From $Z( \mathbf{L}_1 ) \leq Z( \mathbf{L} )$ we obtain
$Z( \mathbf{L}_1 )^\circ \leq Z( \mathbf{L} )^\circ$. Let $\mathbf{T}_d'$ denote the
Sylow $\Phi_d$-torus of $Z( \mathbf{L}_1 )^\circ$. Then $\mathbf{T}_d' \leq \mathbf{T}_d$,
since $\mathbf{T}_d'$ lies in a Sylow $\Phi_d$-torus of $Z( \mathbf{L} )^\circ$, and
$\mathbf{T}_d$ is the unique Sylow $\Phi_d$-torus of $Z( \mathbf{L} )^\circ$;
see \cite[Th{\'e}or{\`e}me~$3.4(3)$]{BrouMa}.
Clearly, $\mathbf{T}_d'$ is nontrivial, as ${D_1 \leq (Z( \mathbf{L}_1 )^\circ)^F}$.
In particular, the rank of $\mathbf{T}_d'$ is at least equal to $\deg( \Phi_d )$.
It follows that $\mathbf{T}_d' = \mathbf{T}_d$.

Let~$d'$ be an integer such that $\Phi_{d'}$ divides the order polynomial of
$(Z( \mathbf{L} )^\circ,F)$. Then $\Phi_{d'}$ also divides the order polynomial 
of $(\mathbf{G},F)$; see \cite[Proposition~$1.11$]{BrouMa}.
By hypothesis, if $p \mid \mathbf{T}_{d'}^F$, then 
${d' = d}$. It follows that $p \nmid 
[(Z( \mathbf{L} )^\circ)^F\colon\!\mathbf{T}_d^F]$ and hence $D \leq \mathbf{T}_d^F$.
Now $D \leq \mathbf{T}_d \leq Z( \mathbf{L} )^\circ$, and thus 
$C_{\mathbf{G}}( D ) \geq C_{\mathbf{G}}( \mathbf{T}_d ) \geq 
C_{\mathbf{G}}( Z( \mathbf{L} )^\circ ) = \mathbf{L} = C_{\mathbf{G}}( D )$, as 
$\mathbf{L}$ is a regular subgroup of~$\mathbf{G}$. We conclude that 
$C_{\mathbf{G}}( D ) = C_{\mathbf{G}}( \mathbf{T}_d )$. Analogously,
$C_{\mathbf{G}}( D_1 ) = C_{\mathbf{G}}( \mathbf{T}_d' )$. Our claim follows
from $\mathbf{T}_d' = \mathbf{T}_d$.
\end{proof}

The result of Lemma~\ref{LargeEll} can be generalized to the case when the 
centralizer of~$D$ is contained in a proper $F$-stable Levi subgroup 
of~$\mathbf{G}$.

\begin{cor}
\label{LargeEllGen}
Let $\mathbf{G}$ denote a connected reductive algebraic group over
$\bar{\mathbb{F}}_r$, and let~$F$ be a Steinberg morphism of~$\mathbf{G}$.
Assume that the positive real number~$q$ associated to $(\mathbf{G},F)$ as in
\cite[Proposition~$1.4.19$]{GeMa} is an integer (which then is a power of~$r$).

Let~$p$ be a prime different from~$r$ and let $D \leq G$ denote a non-trivial 
cyclic $p$-subgroup of~$G$ of order~$p^l$. Suppose that there is an $F$-stable 
Levi subgroup $\mathbf{L} \leq \mathbf{G}$ such 
$C_{\mathbf{G}}( D_1 ) \leq \mathbf{L}$, that~$p$ satisfies the 
condition~{\rm ($\ast$)} with respect to~$(\mathbf{L},F)$ and that $p$ does not 
divide
$|Z(\mathbf{L})/Z(\mathbf{L})^\circ||Z(\mathbf{L}^*)/Z(\mathbf{L}^*)^\circ|$. 
Suppose finally that~$D$ is a radical subgroup of~$L$.
Then $C_{\mathbf{G}}( D ) = C_{\mathbf{G}}( D_i )$ for every $1 \leq i \leq l$.
\end{cor}
\begin{proof}
By Lemma~\ref{LargeEll} applied to~$\mathbf{L}$, we obtain 
$C_{\mathbf{L}}( D ) = C_{\mathbf{L}}( D_1 )$. Hence $C_{\mathbf{G}}( D ) = 
C_{\mathbf{L}}( D ) = C_{\mathbf{L}}( D_1 ) = C_{\mathbf{G}}( D_1 )$.
This implies the assertion for all $1 \leq i \leq l$.
\end{proof}

\section{The quasisimple groups}\label{sec:W(B)quasisimple}

We are now going to address Question~\ref{QuestionA}(c). Throughout this 
section,~$G$ denotes a quasisimple group, ~$S=G/Z(G)$ its simple quotient,
and~$\mathbf{B}$ is a $p$-block of~$G$ with a non-trivial 
cyclic defect group~$D$. 

\subsection{Some particular groups and special cases}
We begin with the case in which~$S$ is a sporadic simple group or 
the Tits simple group.

\begin{prop}
If $S$ is one of the~$26$ simple sporadic groups or the Tits simple group 
$\tw{2}F_4(2)'$, then $W(\bB)\cong k$. 
\end{prop}
\begin{proof}
In all cases $D=D_1$, hence the claim follows from Lemma~\ref{NG(D_1)}(b). 
\end{proof}

Next, we consider the case in which~$S$ is an alternating group.
\begin{prop}
\label{AlternatingGroups}
If~$S$ is the alternating group $\fA_n$ with $n\geq 5$, then $W(\bB)\cong k$.
\end{prop}
\begin{proof}
First assume that $G=\fA_n$ ($n\geq 5$). The $p$-blocks of $\fA_n$, their 
weights and defect groups are, for example, described in 
\cite[Subsection~$6.1$]{JK81} and \cite[Section~$4$]{Olsson90}.

Let $\wh\bB$ be a block of~$\fS_n$ of weight~$w$ covering~$\bB$. A defect
group of~$\wh\bB$ is conjugate to a Sylow $p$-subgroup of $\fS_{wp}$; see 
\cite[Theorem 6.2.45]{JK81}. Thus~$D$ is isomorphic to a Sylow $p$-subgroup 
of~$\fA_{pw}$. If $p$ is odd, this implies $w = 1$ and $D = D_1$, so the claim 
follows from Lemma~\ref{NG(D_1)}(b). As the Sylow $2$-subgroups of $\fA_{2w}$ 
are either trivial or non-cyclic, $\fA_n$ does not have $2$-blocks with 
non-trivial cyclic defect groups. 

Next assume that $|Z(G)|=2$, i.e.\ $G = 2.\fA_n$. We let $\wt{\fS}_n$ denote a
Schur covering group of~${\fS}_n$ containing~$G$ as a subgroup of index~$2$.
For a description of the $p$-blocks of $2.\fA_n$ 
and $\wt{\fS}_n$ we refer to \cite{Olsson90}, \cite{OB}, \cite{Hum86} 
and~\cite{Cab88}.

Suppose that~$p$ is odd. In this case, it only remains to consider the faithful 
blocks, i.e.\ the spin blocks. There is a bijection between the spin blocks of 
$2.\fA_n$ of positive weight and the spin blocks of $\wt{\fS}_n$ of positive 
weight (given by covering), which preserves defect groups. Again, the spin 
blocks of $\wt{\fS}_n$ with cyclic defect groups are precisely the blocks of 
weight $w=1$, which in turn have defect $1$, since the defect groups are 
isomorphic to the Sylow $p$-subgroups of $\wt{\fS}_{pw}$. Hence the claim 
follows from Lemma~\ref{NG(D_1)}(b). 

If $p=2$, then there is a bijection between the $2$-blocks of $2.\fA_n$ and the 
$2$-blocks of $\fA_n$, where defect groups are obtained by moding out the 
center; see \cite[V, Lemma~$4.5$]{Feit}. As we have already observed above, the 
only $2$-blocks of~$\fA_n$ with cyclic defect groups are of defect~$0$. Hence 
$D \leq Z(G)$ and the claim follows from  Lemma~\ref{NG(D_1)}(b). 

Finally, let~$G$ be an exceptional covering group of $\fA_n$ for $n \in \{6,7\}$. 
If $p \geq 5$, we have $D=D_1$, as $p^2 \nmid |G|$. For $p\in\{2,3\}$, we also
get $D = D_1$ using~\cite{ModAT}. Hence the claim follows from 
Lemma~\ref{NG(D_1)}(b) in all cases.
\end{proof}

From now on, we assume that~$S$ is a simple group of Lie type.
\begin{prop}
\label{DefiningCharacteristic}
If~$S$ is a simple group of Lie type in characteristic~$p$,
then $W(\bB)\cong k$.
\end{prop}
\begin{proof}
Let~$\hat{G}$ denote the largest perfect central $p'$-extension of~$S$. 
Then $G/O_p(Z(G))$ is a central quotient of~$\hat{G}$. Now~$D/O_p(Z(G))$ is a
defect group of a $p$-block of $G/O_p(Z(G))$; see \cite[V, Lemma~$4.5$]{Feit}.
Thus $D/O_p(Z(G))$ is a defect group of a $p$-block of~$\hat{G}$ by 
Lemma~\ref{Inflation}.

The claim follows from Proposition~\ref{AlternatingGroups}, if~$S$ is isomorphic 
to an alternating group. We thus exclude the case $S \cong \PSL_2(9) \cong 
\Sp_4(2)' \cong \fA_6$. Then~$\hat{G}$ is of 
the form $\hat{G} = \hat{\bG}^F$ for a connected reductive algebraic 
group ~$\hat{\bG}$ of characteristic~$p$ and a suitable Steinberg morphism~$F$ 
of~$\hat{\bG}$; see \cite[Definition~$6.1.1$(a) and Table~$6.1.3$]{Gor}.
By \cite[Theorem in Section~$8.5$]{Hum06}, the defect groups of the 
$p$-blocks of~$\hat{G}$ are either trivial or the Sylow $p$-subgroups 
of~$\hat{G}$. Thus $D \leq Z(G)$ or~$\hat{G}$ has a 
cyclic Sylow $p$-subgroup. In the former case we obtain our assertion from
Lemma~\ref{NG(D_1)}(b). In the latter case, $S \cong \PSL_2(p)$ and $p \geq 5$.
But then $D = D_1$ and the claim once more follows from Lemma~\ref{NG(D_1)}(b).
\end{proof}

We next consider those cases, where~$S$ has an exceptional Schur multiplier.

\begin{prop}
\label{QuasisimpleLieTypeExceptionalMultiplier}
Assume that~$S$ is isomorphic to one of the~$17$ simple groups of Lie type with 
an exceptional Schur multiplier, listed in \cite[Table~$6.1.3$]{Gor}. (These 
groups constitute~$16$ isomorphism types, as $\PSL_2(9) \cong \Sp_4(2)'$ is 
listed twice.) Then $W( \mathbf{B} ) \cong k$.
\end{prop}
\begin{proof}
As in the proof of Proposition~\ref{DefiningCharacteristic} we exclude the case
$S \cong \Sp_4(2)'$.

Suppose first that~$p$ is odd. We claim that then $|D| = p$, which implies
our assertion using Lemma~\ref{NG(D_1)}(b). If $S \cong {^2\!E}_6(2)$ and 
$p \leq 5$, the claim can be checked with~\cite{GAP04}, and in all other cases 
with \cite{ModAT}. 

Now assume that $p = 2$. If~$S$ has characteristic~$2$, the assertion follows
from Proposition~\ref{DefiningCharacteristic}. Otherwise,~$S$ is isomorphic to
$\PSU_4(3)$, $\Omega_7(3)$ or $G_2(3)$. In the first and third case, we find
$|D| = p$ by inspecting \cite{ModAT}. It remains to consider the
case $S \cong \Omega_7(3)$ which we handled with~\cite{GAP04}. We distinguish
the cases $G = S$ and ${G = \Spin_7( 3 ) = 2.S}$. If $G = S$,
there is a unique $2$-block of~$G$ with non-trivial cyclic defect group, and 
this has order $2$. If $G = \Spin_7( 3 )$, there are three $2$-blocks of~$G$
with non-trivial cyclic defect groups, two of order~$2$ and one of order~$4$.
The claim follows from Lemma~\ref{NG(D_1)}(b)(c).
\end{proof}

\subsection{Groups of Lie type and large primes}
In the considerations below we adopt the following common
convention. Let $\varepsilon \in \{ 1, -1 \}$. Then $\PSL^\varepsilon_n(q)$
denotes the projective special linear group $\PSL_n(q)$, if $\varepsilon = 1$,
and the projective special unitary group $\PSU_n(q)$, if $\varepsilon = -1$.
Analogous conventions are used for $\SL^\varepsilon_n(q)$ and
$\GL^\varepsilon_n(q)$.

\begin{prop}
\label{QuasisimpleLieTypeNonExceptionalMultiplier}
Let~$G$ denote a quasisimple group of Lie type with simple quotient~$S$.
Assume that~$S$ does not have an exceptional Schur multiplier and let 
$\hat{G}$ denote the universal covering group of~$G$. Let~$p$
be an odd prime different from the defining characteristic~$r$ of~$S$. 

Let~$\mathbf{B}$ be a $p$-block of~$G$ with a non-trivial cyclic defect
group~$D$ of order $p^l$. Then $C_G( D ) = C_G( D_i )$ and 
$N_G( D ) = N_G( D_i )$ for all $1 \leq i \leq l$ under any of the conditions 
in~{\rm (a)} or~{\rm (b)} below. In particular, $W( \mathbf{B} ) \cong k$ if 
these conditions are satisfied.

{\rm (a)} Suppose that~$S$ is a classical group. Then~$\hat{G}$ is one of 
$\SL_n(q)$ with 
$n \geq 2$, $\SU_n(q)$ with $n \geq 3$, $\Sp_n(q)$ with $n \geq 4$ even, 
$\Spin_n(q)$ with $n \geq 7$ odd, or $\Spin^{\pm}_n(q)$ with $n \geq 8$ 
even.

Let~$d$ denote the order of~$q$ modulo~$p$. If~$\hat{G}$ is one of $\SL_n(q)$, 
$\Sp_n(q)$ or $\Spin^-_n(q)$, assume that $pd > n$. If $\hat{G} = \Spin_n( q )$ 
with $n$ odd, assume that $pd > n - 1$. If $\hat{G} = \Spin^+_n(q)$ with~$n$ 
even, assume that $pd > n - 2$. Finally, if $\hat{G} = \SU_n(q)$, assume that 
$pd > 2n$.

{\rm (b)} Suppose that~$S$ is an exceptional group of Lie type, including
the Suzuki and Ree groups. Assume that 
$p > 3$, and if~$S$ is of type~$E_8$ assume that $p > 5$.
\end{prop}
\begin{proof}
It suffices to prove the first assertion, i.e.\ that 
$C_G( D ) = C_G( D_i )$ for all $1 \leq i \leq l$.
Indeed, Lemma~\ref{NG(D_1)}(b) then implies that $W( \mathbf{B} ) \cong k$.
Moreover, $N_G( D ) \leq N_G( D_i ) \leq N_G( C_G( D_i ) ) = N_G( C_G( D ) ) 
= N_G( D )$ for all $1 \leq i \leq l$, where the last equality arises from
Lemma~\ref{AbelianRadicalSubgroup}. 

Let us now prove the first assertion.
To begin with, assume that~$S$ is a Suzuki group $S = {^2\!B}_2(2^{2m+1})$ for 
some $m \geq 2$, in which case $G = S$. It follows from the proof of
\cite[Theorem~$4$]{Suzuki}, that the centralizer of every non-trivial
$p$-element of~$S$ is a maximal torus. This implies our assertion. An analogous
argument applies if~$S$ is a Ree group $S = {^2G}_2(3^{2m+1})$ for some 
$m \geq 1$, using \cite{WardRee}. If $S = {^2\!F}_4(2^{2m+1})$ for some 
$m \geq 1$, we again have $G = S$. The conjugacy classes of~$S$ are determined 
in~\cite{shino}. Let $D = \langle t \rangle$; then~$t$ is a semisimple element 
of~$S$ and, by assumption, $3 \nmid |t|$. Now~$t$ is contained in one of the 
maximal tori $T(i)$, $i = 1, \ldots, 11$ of~$S$, whose structure is given in
\cite[p.~$19$]{shino}. Suppose first that~$t$ is a regular element of $T(i)$,
i.e.\ $C_S( t ) = T( i )$. As~$D$ is a radical $p$-subgroup of~$G$, the Sylow
$p$-subgroup of $T(i)$ is cyclic. Hence $i \in \{ 5, 9, 10, 11 \}$. Indeed, a
$p$-element in $T(2) \cup T(3) \cup T(4)$ is not regular, and the
Sylow $p$-subgroups of $T(1)$, $T(6)$, $T(7)$ or $T(8)$ are not cyclic.
From \cite[Table~IV]{shino} we conclude that every non-trivial power of~$t$
also is regular in $T(i)$, proving our claim in the first case. Now assume
that~$t$ is not a regular element. Then~$t$ is conjugate to one of~$t_j$,
with $j \in \{ 1, 2, 5, 7, 9 \}$ in the notation of \cite[Table~IV]{shino}.
This table then shows that every non-trivial power of~$t$ has the same
centralizer as~$t$.

We will henceforth assume that~$S$ is not one of the Suzuki or Ree groups. By 
our assumption,~$p$ does not divide the order of the Schur multiplier of~$S$. 
Thus, by Lemma~\ref{CentralExtensionEtc}(a), we may assume that $G = \hat{G}$ is 
the universal central extension of~$S$. 
Thus there is a simple, simply connected algebraic group $\mathbf{G}$ over 
$\bar{\mathbb{F}}_r$, and a Steinberg morphism~$F$ of~$\mathbf{G}$, such that 
$G = \mathbf{G}^F$. Moreover, $\mathbf{G}$ and~$F$ may be chosen such that the
number~$q$ associated to $(\mathbf{G},F)$ as in 
\cite[Proposition~$1.4.19$(b)]{GeMa} is an integer. In fact it agrees with the 
number~$q$ introduced in the hypotheses~(a). Now let~$d$ denote the order of~$q$
modulo~$p$. Then $p \mid \Phi_d(q)$, and thus $\Phi_d$ divides the order 
polynomial of the generic finite reductive group~$\mathbb{G}$ associated with 
$(\mathbf{G},F)$; 
see \cite[Corollary~$5.4$]{MalleHeight0}. Moreover, if $d'$ is an integer such 
that $p \mid \Phi_{d'}(q)$, then $d' = p^jd$ for some non-negative integer~$j$; 
see \cite[Lemma~$5.2$(a)]{MalleHeight0}. The order polynomial of~$\mathbb{G}$ 
can be read off from the order formulae for~$\hat{G}$ given, for example in 
\cite[Table~$1.3$]{GeMa}. We conclude that all hypotheses of 
Lemma~\ref{LargeEll} on~$\mathbf{G}$,~$F$ 
and~$p$ are satisfied. Our claim follows from Lemma~\ref{LargeEll}.
\end{proof}

\section*{Acknowledgements}

The authors thank Olivier Dudas, Meinolf Geck, Markus Linckelmann, Michael
Livesey, Frank L{\"u}beck and Gunter Malle for helpful discussions 
and suggestions.

\end{document}